\documentclass[10pt]{article}
\usepackage{mathrsfs}
\usepackage{mathrsfs}
\usepackage{amsmath}
\usepackage{lineno}
\usepackage{t1enc}
\usepackage[latin1]{inputenc}
\usepackage[english]{babel}
\usepackage{amsmath,amsthm}
\usepackage{amsfonts}
\usepackage{latexsym}
\usepackage[dvips]{graphicx}
\usepackage{graphicx}
\usepackage[natural]{xcolor}
\newtheorem{theorem}{Theorem}[section]
\newtheorem{lemma}[theorem]{Lemma}

\newtheorem{conjecture}[theorem]{Conjecture}

\newtheorem{corollary}[theorem]{Corollary}

\newtheorem{remark}[theorem]{Remark}

\makeatletter

\newcommand{\Rmnum}[1]{\expandafter\@slowromancap\romannumeral #1@}
\newcommand{\col}{\varphi}
\makeatother

\begin{document}

\title{\LARGE Edge Covering Pseudo-outerplanar Graphs with Forests\footnotetext{E-mails: sdu.zhang@yahoo.com.cn (X. Zhang),
gzliu@sdu.edu.cn (G. Liu), jlwu@sdu.edu.cn (J. L. Wu).}\thanks{Research supported by NSFC (10971121, 11101243, 61070230), RFDP (20100131120017) and GIIFSDU (yzc10040).}
}
\author{Xin Zhang,~~Guizhen Liu\thanks {Corresponding author.}
 ~~and Jian-Liang Wu
\\
{\small School of Mathematics, Shandong University, Jinan, 250100,
China}}
\date{}

\maketitle

\begin{abstract}
\baselineskip 0.50cm

A graph is called pseudo-outerplanar if each block has an embedding on the plane in such a way that the vertices lie on a fixed circle and the edges lie inside the disk of this circle with each of them crossing at most one another. In this paper, we prove that each pseudo-outerplanar graph admits edge decompositions into a linear forest and an outerplanar graph, or a star forest and an outerplanar graph, or two forests and a matching, or $\max\{\Delta(G),4\}$ matchings, or $\max\{\lceil\Delta(G)/2\rceil,3\}$ linear forests.
These results generalize some ones on outerplanar graphs and $K_{2,3}$-minor-free graphs, since the class of pseudo-outerplanar graphs is a larger class than the one of $K_{2,3}$-minor-free graphs.

\vspace{2mm}\noindent\textit{Keywords}: pseudo-outerplanar graphs; edge decomposition; edge chromatic number; linear arboricity.
\end{abstract}

\baselineskip 0.50cm

\section{Introduction}

In this paper, all graphs considered are finite, simple and undirected. We use $V(G)$, $E(G)$, $\delta(G)$ and $\Delta(G)$ to denote the vertex set, the edge set, the minimum degree and the maximum degree of a graph $G$, respectively. Let $d_G(v)$ (or $d(v)$ for simplicity) denote the \emph{degree} of a vertex $v\in V(G)$. A $block$ is a maximal 2-connected subgraph of a given graph $G$. A graph $H$ is a $minor$ of a graph $G$ if a copy of $H$ can be obtained from $G$ via repeated edge deletion and/or edge contraction. For a subset $S\subseteq V(G)\cup E(G)$, $G[S]$ denotes the subgraph of $G$ induced by $S$. The \emph{vertex connectivity} of a graph $G$, denoted by $\kappa(G)$, is the minimum number of vertices whose deletion from $G$ disconnects it.
For other undefined concepts we refer the readers to \cite{Bondy}.

An \emph{outerplanar graph} is a graph that can be embedded on the plane in such a way that it has no crossings and that all its vertices lie on the outer face. In this paper, we aim to introduce an extension of this concept. A graph is called \emph{pseudo-outerplanar} if each block has an embedding on the plane in such a way that the vertices lie on a fixed circle and the edges lie inside the disk of this circle with each of them crossing at most one another. In this embedding, the edges bounding the disk(s) are called \emph{boundary edges} and a disk is said to be \emph{closed} or \emph{open} according to whether or not it contains the circle that constitutes its boundary.
For example, Figure \ref{K4K23} exhibits a pseudo-outerplanar embedding of a graph with two blocks: one is $K_4$ and the other
is $K_{2,3}$. The drawing of $K_4$ in this embedding lies inside a closed disk but the one
of $K_{2,3}$ in this embedding lies inside an open disk. In Figure \ref{K4K23}, the edges in bold are the boundary edges.
A pseudo-outerplanar graph is \emph{maximal} if it is not possible to add an edge such that the resulting graph is still pseudo-outerplanar. Thus $K_{2,3}$ is not a maximal pseudo-outerplanar graph, since we can possibly add two edges to $K_{2,3}$ and remain its pseudo-outerplanarity. One can easily check that each pseudo-outerplanar graph has a planar embedding by its definition. So the class of pseudo-outerplanar graphs forms a subclass of planar graphs. Actually, the definition of pseudo-outerplanar graphs are similar to that of 1-planar graphs (i.e. graphs that can be drawn on the plane so that each edge is crossed by at most one other edge), which was introduced by Ringel \cite{Ringel}.
\begin{figure}
\begin{center}
  \includegraphics[width=7cm,height=3cm]{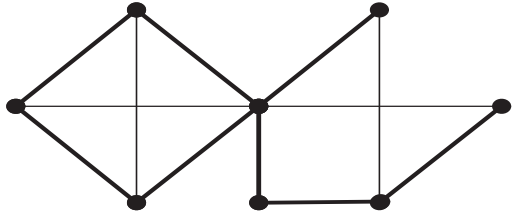}\\
  \caption{An example of pseudo-outerplanar}
  \label{K4K23}
\end{center}
\end{figure}

Many classic problems in graph theory are considered for the class of planar graphs and its subclasses, such as the class of series-parallel graphs and the one of outerplanar graphs. Taking the problem of covering graphs with forests and a graph of bounded maximum degree for example, we say that a graph is $(t,d)$-$coverable$ if its edges can be covered by at most $t$ forests and a graph of maximum degree $d$. In \cite{Balogh}, Balogh et al. conjectured that every simple planar graph is $(2,4)$-coverable and gave a example to show that there are infinitely many planar graphs that are not $(2,3)$-coverable. This conjecture was recently confirmed by Gonçalves in \cite{Goncalves}. In \cite{Balogh}, it is also proved that every series-parallel graph is $(2,0)$-coverable and that every $K_{2,3}$-minor-free graph is both $(1,3)$-coverable and $(2,0)$-coverable. Since a graph is outerplanar if and only if it is $\{K_4,K_{2,3}\}$-minor-free \cite{Harary}, every outerplanar graph is both $(1,3)$-coverable and $(2,0)$-coverable. It is interesting to know what can be said about pseudo-outerplanar graphs, another larger class than outerplanar
graphs.

Edge-coloring is another classic problem in graph theory. In fact, we can regard edge-coloring problems as a covering problem. When we color the edges of a graph $G$, our actual task is to decompose the edge set $E(G)$ into some parts such that the graph induced by each part satisfies a property $\mathcal{P}$. Different properties $\mathcal{P}$ correspond to different types of edge-coloring. For example, a \emph{proper $k$-edge-coloring} of $G$ is a decomposition of $E(G)$ into $k$ subsets such that the graph induced by each subset is a matching in $G$. The minimum integer $k$ such that $G$ has a proper $k$-edge-coloring, denoted by $\chi'(G)$, is the \emph{edge chromatic number} of $G$. Vizing's Theorem states that for any graph $G$, $\Delta(G)\leq \chi'(G)\leq\Delta(G)+1$. A graph $G$ is said to be of \emph{class 1} if $\chi'(G)=\Delta(G)$, and of \emph{class 2} if $\chi'(G)=\Delta(G)+1$. To determine whether a planar graph is of class 1 is an interesting problem. Sanders and Zhao \cite{Sanders} showed that each planar graph with maximum degree at least 7 is of class 1. Juvan, Mohar and Thomas \cite{Juvan} proved that each series-parallel graph with maximum degree at least 3 is of class 1, and thus holds for outerplanar graphs. It is open whether each pseudo-outerplanar graph with large maximum degree is of class 1.

On the other hand, one can consider improper edge-colorings. Concerning this topic, Harary \cite{Harary_la} introduced the concept of linear arboricity. A \emph{linear forest} is a forest in which every connected component is a path. A \emph{$k$-tree-coloring} of $G$ is a decomposition of $E(G)$ into $k$ subsets such that the graph induced by each subset is a linear forest.
The \emph{linear arboricity} $la(G)$ of
a graph $G$ is the minimum integer $k$ such that $G$ has a $k$-tree-coloring. Akiyama, Exoo and Harary \cite{Akiyama} conjectured that $la(G)=\lceil(\Delta(G)+1)/2\rceil$ for any regular graph $G$. It is obvious that $la(G)\geq\lceil\Delta(G)/2\rceil$ for any graph $G$ and $la(G)\geq \lceil(\Delta(G)+1)/2\rceil$ for any regular graph $G$. Hence the conjecture is equivalent to the following one.

\begin{conjecture}[Linear Arboricity Conjecture] \label{lac} For any graph $G$, $\lceil\frac{\Delta(G)}{2}\rceil\leq la(G)\leq\lceil\frac{\Delta(G)+1}{2}\rceil.$
\end{conjecture}

\noindent Now Conjecture \ref{lac} has been proved true for all planar graphs (see \cite{jlWu,WW}). However, it is still interesting to determine which kinds of planar graphs satisfy $la(G)=\lceil\Delta(G)/2\rceil$. Wu \cite{jlWu} proved that it holds for planar graphs with maximum degree at least 13. And the bound 13 was later improved to 9 by Cygan et al. \cite{Cygan}. For subclasses of planar graphs, Wu \cite{Wu} proved that $la(G)=\lceil\Delta(G)/2\rceil$ for all series-parallel graphs (hence also for all outerplanar graphs) with maximum degree at least 3. Can the same conclusion extend to the class of pseudo-outerplanar graphs?

In Section 2, we give some relationships among three classes containing the outerplanar graphs; they are the $K_{2,3}$-minor-free graphs, the series-parallel graphs and the pseudo-outerplanar graphs. In Section $3$, we investigate the problem of covering pseudo-outerplanar graphs with forests and a graph of bounded maximum degree. In Section 4, some unavoidable structures of pseudo-outerplanar graphs are obtained. These structures will be applied to determine the edge chromatic number and linear arboricity of pseudo-outerplanar graphs in Section 5.

\section{Basic Properties}

Let $G$ be a pseudo-outerplanar graph. In the following of this paper, we always assume that $G$ has been drawn on the plane such that (1) for each block $B$ of $G$, the vertices of $B$ lie on a fixed circle and the edges of $B$ lie inside the disk of this circle with each of them crossing at most one another; (2) the number of crossings in $G$ is as small as possible. This drawing is called a $pseudo$-$outerplanar$ $diagram$ of $G$. Let $G$ be a pseudo-outerplanar diagram and let $B$ be a block of $G$. Denote by $v_1, v_2, \cdots, v_{|B|}$ the vertices of $B$, which are lying in a clockwise sequence. Let $\mathcal{V}[v_i,v_j]=\{v_i, v_{i+1}, \cdots, v_j\}$ and $\mathcal{V}(v_i,v_j)=\mathcal{V}[v_i,v_j]\backslash \{v_i,v_j\}$, where the subscripts and the additions are taken modular $|B|$.

\begin{lemma} {\rm \cite{Harary}} \label{opg}
Let $G$ be an outerplanar graph. Then

\noindent{\rm(a)} $\delta(G)\leq 2$,

\noindent{\rm(b)} $\kappa(G)\leq 2$.
\end{lemma}

\begin{theorem} \label{degree}
Let $G$ be a pseudo-outerplanar graph. Then

\noindent{\rm(a)} $\delta(G)\leq 3$,

\noindent{\rm(b)} $\kappa(G)\leq 2$ unless $G\simeq K_4$.
\end{theorem}

\begin{proof}
The proof of (a) is left to Corollary \ref{mindegree3}. So we only prove (b) here. If $|G|\leq 4$, then this theorem is trivial. So we assume that $G$ is a pseudo-outerplanar diagram with $|G|\geq 5$ and $\kappa(G)\geq 3$. If $G$ has no crossings, then $G$ is an outerplanar graph and thus by Lemma \ref{opg}, $\kappa(G)\leq 2$, a contradiction. So we assume that there are two chords $v_iv_j$ and $v_kv_l$ in $G$ that cross each other, and that $v_i,v_k,v_j,v_l$ are lying in a clockwise sequence. Since $|G|\geq 5$, at least one of $\mathcal{V}(v_i,v_k)$, $\mathcal{V}(v_k,v_j)$, $\mathcal{V}(v_j,v_l)$ and $\mathcal{V}(v_l,v_i)$ is nonempty. Without loss of generality, assume that $\mathcal{V}(v_i,v_k)\neq \emptyset$. Since $v_iv_j$ crosses $v_kv_l$, there is no edges between the two vertex sets $\mathcal{V}(v_i,v_k)$ and $\mathcal{V}(v_k,v_i)$. So $\{v_i,v_k\}$ separates $\mathcal{V}(v_i,v_k)$ and $\mathcal{V}(v_k,v_i)$, contradicting to $\kappa(G)\geq 3$.
\end{proof}

It is well-known that every $2$-connected outerplanar graph is hamiltonian. But this result does not hold for $2$-connected pseudo-outerplanar graphs. The complete bipartite graph $K_{2,3}$ is such a counterexample. A $2$-connected pseudo-outerplanar diagram is called a $hamiltonian$ $diagram$ if it is in such a way that all its vertices lie on a \emph{closed circuit} $C$ (i.e. the disk of $C$ is closed). This closed circuit $C$ is called the $hamiltonian$ $boundary$ of the diagram. By this definition, one can easily see that a non-hamiltonian $2$-connected pseudo-outerplanar graph cannot have a hamiltonian diagram. It seems interesting to answer whether each hamiltonian pseudo-outerplanar graph has a hamiltonian diagram.
\begin{figure}
\begin{center}
  \includegraphics[width=16.8cm,height=4cm]{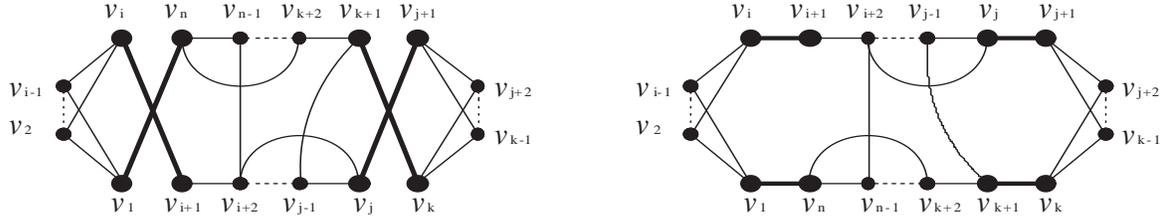}\\
  \caption{Each hamiltonian pseudo-outerplanar graphs has a hamiltonian diagram}
  \label{hamilton}
\end{center}
\end{figure}

\begin{theorem}
Let $G$ be a pseudo-outerplanar diagram and $C$ be a hamiltonian cycle of $G$. If $C$ is not the boundary of $G$, then $G$ has a hamiltonian diagram such that $C$ is the hamiltonian boundary of this diagram.
\end{theorem}

\begin{proof}
We proceed by induction on the order of $G$. Since $G$ has a hamiltonian cycle $C=v_1v_2\cdots v_nv_1$ that is not the boundary of the pseudo-outerplanar diagram of $G$, one can easily deduce that there exists at least one crossing in the drawing of $C$ (a sub-diagram of $G$ indeed).
Suppose that $v_jv_{j+1}$ and $v_kv_{k+1}$ (j<k) cross each other and that $v_j$ follows $v_k$ in a clockwise walk around $G$. Denote respectively by $U$ and $W$ the set of vertices from $v_j$ to $v_{k+1}$ and from $v_{j+1}$ to $v_{k}$ in the cyclic clockwise sequence of vertices on the outer boundary of $G$. Take the first graph in Figure \ref{hamilton} for example, we have $C=v_1v_2\cdots v_nv_1$, $U=\{v_j,v_{j-1},\cdots,v_{i+1},v_1,\cdots,v_i,v_n,v_{n-1},\cdots,v_{k+1}\}$ and $W=\{v_{j+1},v_{j+2},\cdots,v_{k-1}, v_{k}\}$. Note that besides $v_jv_{j+1}$ and $v_kv_{k+1}$, there is no other edge $uw$ such that $u\in U$ and $w\in W$ by the definition of $G$. Let $G_1=G[U]+v_jv_{k+1}$ and $G_2=G[W]+v_{j+1}v_k$. Then $G_1$ is a pseudo-outerplanar diagram with a hamiltonian cycle $C_1=v_{k+1}v_{k+2}\cdots v_nv_1\cdots v_jv_{k+1}$ while $G_2$ is a pseudo-outerplanar diagram with a hamiltonian cycle $C_2=v_{j+1}v_{j+2}\cdots v_{k}v_{j+1}$. By induction hypothesis, $G_1,G_2$ respectively has a hamiltonian diagram such that $C_1,C_2$ is the hamiltonian boundary of each diagram. Now we combine these two hamiltonian diagrams and add two edges $v_jv_{j+1}$ and $v_kv_{k+1}$ (see the second graph in Fig.\ref{hamilton}), then we can get a hamiltonian diagram of $G$ with hamiltonian boundary $v_{k+1}v_{k+2}\cdots v_nv_1\cdots v_jv_{j+1}v_{j+2}\cdots v_{k-1}v_{k}v_{k+1}$, which is the cycle $C$ indeed.
\end{proof}

\begin{corollary}\label{hamcor}
Each hamiltonian pseudo-outerplanar graph has a hamiltonian diagram.
\end{corollary}

We say a graph $G$ $quasi$-$hamiltonian$ if each block of $G$ is hamiltonian. Denote the class of pseudo-outerplanar graphs, quasi-hamiltonian pseudo-outerplanar graphs, series-parallel graphs, $K_{2,3}$-minor-free graphs and outerplanar graphs by $\mathcal{P}$, $\mathcal{P}_H$, $\mathcal{S}$, $\mathcal{M}_{2,3}$ and $\mathcal{O}$, respectively. The following basic relationship is obvious.

\begin{remark} \label{basic}
$\mathcal{P}\supset\mathcal{P}_H\supset\mathcal{O}$,
$\mathcal{M}_{2,3}\bigcap \mathcal{S}=\mathcal{O}$
\end{remark}

In the following, we continue to study some more interesting
relationships among these five classes of graphs.

\begin{theorem}
$\mathcal{P}_H\bigcap \mathcal{S}=\mathcal{O}$.
\end{theorem}

\begin{proof}
Let $G\in \mathcal{P}_H\bigcap \mathcal{S}$ and let $B$ be a block of $G$. By Corollary \ref{hamcor} $B$ has a hamiltonian diagram, and actually this diagram is outerplanar. If there was a crossing, there would be four vertices $u,v,x,y$ with $uv$ and $xy$ crossing in $B$. Since the diagram is hamiltonian, there are four pairwise disjoint paths $P_{ux},P_{xv},P_{vy}$ and $P_{yu}$ that connects $u$ to $x$, $x$ to $v$, $v$ to $y$ and $y$ to $u$. Thus the edges $uv$ and $vy$ and the four paths $P_{ux},P_{xv},P_{vy}, P_{yu}$ form a $K_4$-minor, which is impossible in a series-parallel graph. Hence $B$ is an outerplanar graph.
\end{proof}

\begin{lemma}{\rm \cite{Hetherington}}\label{near}
Let $H$ be a graph obtained from $K_{2,3}$ by adding an edge joining two vertices of degree $2$ and let $G$ be a $H$-minor-free graph. Then each block of $G$ is
either $K_4$-minor-free or isomorphic to $K_4$.
\end{lemma}

\begin{corollary} \label{lem}
For any $2$-connected graph $G\in \mathcal{M}_{2,3}$, either $G\in \mathcal{O}$ or $G\simeq K_4$.
\end{corollary}

\begin{proof}
Since $G\in \mathcal{M}_{2,3}$, $G$ is $H$-minor-free where $H$ is the graph in Lemma \ref{near}. Thus by Remark \ref{basic} and Lemma \ref{near} either $G\in \mathcal{O}$ or $G\simeq K_4$.
\end{proof}

\begin{theorem}
$\mathcal{M}_{2,3}\subset \mathcal{P}_H$.
\end{theorem}

\begin{proof}
The inclusion of $\mathcal{M}_{2,3}$ in $\mathcal{P}_H$ directly follows from Corollary \ref{lem}. The inequality comes from the graph $(K_1\bigcup K_2)\vee \overline{K_2}$ that belongs to $\mathcal{P}_H$ but not to $\mathcal{M}_{2,3}$.
\end{proof}

\section{Decomposability}

Let $G$ be a pseudo-outerplanar diagram and let $B$ be a block of $G$. Denote by $v_1, v_2, \cdots, v_{|B|}$ the vertices of $B$, which are lying in a clockwise sequence. The edges of the form $v_iv_{j}~(j-i=1~{\rm or}~|B|-1)$ are called $boundaries$ while the edges of the form $v_iv_j~(1<j-i<|B|-1)$ are called $chords$ of $G$. Since $G$ is a pseudo-outerplanar diagram, all the crossings are generated by one chord crossing another chord. Let $\mathcal{C}[v_i,v_j]$ be the set of chords $xy$ with $x,y\in \mathcal{V}[v_i,v_j]$ and let $\mathcal{C}(G)$ be the set of crossed chords in $G$.

\begin{figure}
\begin{center}
  \includegraphics[width=15cm,height=3.5cm]{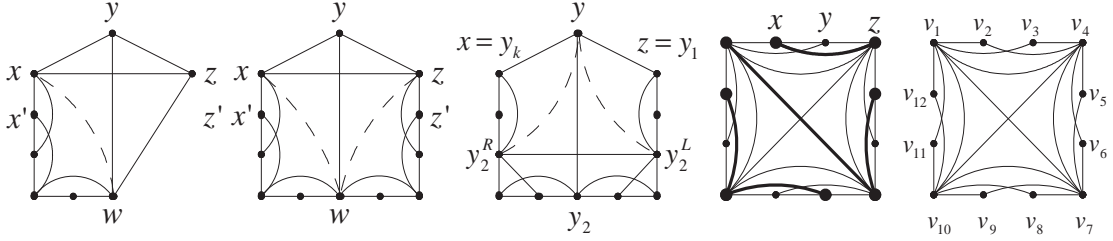}\\
  \caption{Decomposability of pseudo-outerplanar graphs}
  \label{covering}
\end{center}
\end{figure}

\begin{theorem} \label{linear}
Let $G$ be a hamiltonian pseudo-outerplanar diagram and $C$ be the hamiltonian boundary of this diagram. Let $y\in V(C)$ and $yx,yz\in E(C)$. Then there exists a linear forest $T$ in $G$ such that $E(T)\subseteq\mathcal{C}(G)$, $d_T(y)=0$, $\max\{d_T(x),d_T(z)\}\leq 1$, and $G-E(T)$ is an outerplanar diagram.
\end{theorem}

\begin{proof}
We proceed by induction on the order of $G$. One can see that the theorem holds for $|G|\leq 4$ since the case $G=K_4$ is trivial. So we assume that $|G|\geq 5$. In the following, we also assume that the three vertices $x,y,z$ occur on $C$ in a clockwise sequence.

First, we consider the case when $d_G(y)=2$. Let $G'=(G-y)+xz$ and $C'=(C-y)+xz$ (note that if the edge $xz$ already exists in $G$, we let $G'=G-y$ and $C'=C-y$). Then $G'$ is a hamiltonian pseudo-outerplanar diagram with $C'$ being its hamiltonian boundary. Let $xx'\in E(C')$ with $x'\not=z$ ($x'$ exists because $|V(G)|\geq 5$). By induction on $(G',C',x',x,z)$ (as $(G,C,x,y,z)$, respectively), there exists a linear forest $T'$ in $G'$ such that $E(T')\subseteq\mathcal{C}(G')$, $d_{T'}(x)=0$, $\max\{d_{T'}(x'),d_{T'}(z)\}\leq 1$, and $G'-E(T')$ is an outerplanar diagram.
Note that $\mathcal{C}(G')=\mathcal{C}(G)$. Let $T=T'$. Then $E(T)\subseteq \mathcal{C}(G)$, $d_T(x)=d_T(y)=0$ and $d_T(z)\leq 1$. Furthermore, one can easily see that $G-E(T)$ is an outerplanar diagram.

If $d_G(y)=3$ and $xz\in E(G)$, then the edge $xz$ is crossed by another edge $yw$. Assume first that $\mathcal{V}(z,w)=\emptyset$, then $zw\in E(C)$. Let $G'=G[\mathcal{V}[w,x]]+wx$ and let $C'$ be the cycle consisting of the edge $xw$ and the clockwise subpath around $C$ from $w$ to $x$. We assume that $N_{C'}(x)\setminus \{w\}\not=\emptyset$, because otherwise $G$ would have less than five vertices, a contradiction.
Let $xx'\in E(C')$ with $x'\not=w$ (see 1st graph of Figure \ref{covering}). Note that $G'$ is a hamiltonian pseudo-outerplanar diagram with $C'$ being its hamiltonian boundary. By induction on $(G',C',x',x,w)$, there exists a linear forest $T'$ in $G'$ such that $E(T')\subseteq\mathcal{C}(G')$, $d_{T'}(x)=0$, $\max\{d_{T'}(x'),d_{T'}(w)\}\leq 1$, and $G'-E(T')$ is an outerplanar diagram. Let $T=T'+xz$. Then $E(T)\subseteq\mathcal{C}(G)$, $d_T(y)=0$, $d_T(x)=d_T(z)=1$, and $G-E(T)$ is an outerplanar diagram. Thus a linear forest $T$ as required has been constructed. So in the following, we assume that $\mathcal{V}(z,w)\neq\emptyset$ and $\mathcal{V}(w,x)\neq\emptyset$.
Let $zz'\in E(C_1)$ with $z'\not=y,w$, and let $xx'\in E(C)$ with $x'\not=y,w$ (see 2nd graph of Figure \ref{covering}).
Set $G_1=G[\mathcal{V}[z,w]]+zw$ and $G_2=G[\mathcal{V}[w,x]]+wx$. By $C_1$ and $C_2$, we respectively denote the cycle that consists of the edge $wz$ and the clockwise subpath around $C$ from $z$ to $w$, and that consists of the edge $xw$ and the clockwise subpath around $C$ from $w$ to $x$. Then for $i=1,2$, $G_i$ is a hamiltonian pseudo-outerplanar diagram with $C_i$ being its hamiltonian boundary.
By inductions on $(G_1,C_1,w,z,z')$ and $(G_2,C_2,w,x,x')$, there respectively exists a linear forest $T_1$ in $G_1$ with $E(T_1)\in \mathcal{C}(G_1)$, $d_{T_1}(z)=0$, $\max\{d_{T_1}(w),d_{T_1}(z')\}\leq 1$ and $G_1-E(T_1)$ being an outerplanar diagram, and a linear forest $T_2$ in $G_2$ with $E(T_2)\in \mathcal{C}(G_2)$, $d_{T_2}(x)=0$, $\max\{d_{T_2}(w),d_{T_2}(x')\}\leq 1$ and $G_2-E(T_2)$ being an outerplanar diagram.
Let $T=T_1 \cup T_2 \cup \{xz\}$. Then we can easily see that $E(T)\subseteq\mathcal{C}(G)$, $d_T(y)=0$, $d_T(x)=d_T(z)=1$, $d_T(w)\leq 2$ and $G-E(T)$ is an outerplanar diagram.
Since $T_1$ and $T_2$ intersect on at most one vertex, $w$, of degree at most one in each forest and there is no edges between $V(T_1)\setminus \{w\}$ and $V(T_2)\setminus \{w\}$, $T_1\cup T_2$ is a linear forest. Furthermore since $x, y$ and $z$ have degree $0$ in $T_1\cup T_2$, $T_1 \cup T_2 \cup \{xz\}$ is as required.

The last case is when $d_G(y)\geq 3$ and $xz\not\in E(G)$. We label the neighbors of $y$ by $y_1,y_2,\cdots,y_k$ in a clockwise sequence on $C$, where $y_1=z$, $y_k=x$ and $k\geq 3$. If $yy_2$ is not a crossed chord in $G$, then set $G_1=G[\mathcal{V}[y,y_2]]$ and $G_2=G[\mathcal{V}[y_2,y]]$. Denote by $C_1$ (resp. $C_2$) the cycle consisting of the edge $yy_2$ and the clockwise subpath around $C$ from $y$ to $y_2$ (resp. from $y_2$ to $y$). Then $G_i$($i=1,2$) is a hamiltonian pseudo-outerplanar diagram with $C_i$ being its hamiltonian boundary. By using inductions on $(G_1,C_1,y_2,y,z)$ and $(G_2,C_2,y_2,y,x)$,
it is easy to construct a linear forest as required. So we assume that $yy_2$ is crossed by another edge $y_2^Ly_2^R$ in $G$, where $y_2^L, y_2, y_2^R$ are labeled clockwise. Since there is no edges between $\mathcal{V}(y,y_2^L)$ and $\mathcal{V}(y_2^L,y)$, or between $\mathcal{V}(y,y_2^R)$ and $\mathcal{V}(y_2^R,y)$, we can add two edges $yy_2^L$ and $yy_2^R$ to $G$ if they do not really exist so that they do not generate new crossings in $G$ and thus the resulting graph is still pseudo-outerplanar (see the 3rd graph of Fig. \ref{covering}).
By $C_1$, $C_2$ and $C_3$, we respectively denote the cycle that consists of the edge $y_2^Ly$ and the clockwise subpath around $C$ from $y$ to $y_2^L$, and that consists of the path $y_2^R y y_2^L$ and the clockwise subpath around $C$ from $y_2^L$ to $y_2^R$, and that consists of the edge $yy_2^R$ and the clockwise subpath around $C$ from $y_2^R$ to $y$. Let $G_i$ be the subgraph of $G$ contained in the closed disc of $C_i$($i=1,2,3$). Here one should be careful that if $y_2^L=y_1$ (resp. $y_2^R=y_k$), then $C_1$ (resp. $C_3$) is not a cycle indeed and then $G_1$ (resp. $G_3$) is defined to be a null graph. However, $G_1$ and $G_3$ cannot simultaneously be null graphs, since $y_1y_k \not\in E(G)$. Hence any of $G_i$($i=1,2,3$) is a subgraph of $G$ with smaller order. Moreover, every non-null graph $G_i$ is a hamiltonian pseudo-outerplanar diagram with $C_i$ being its hamiltonian boundary.
Without loss of generality, we assume that none of $G_i$($i=1,2,3$) is null graph.
By inductions on $(G_1,C_1,y_1,y,y_2^L)$, $(G_2,C_2,y_2^R,y,y_2^L)$ and $(G_3,C_3,y_k,y,y_2^R)$, there exists a linear forest $T_i$ in $G_i$ such that $E(T_i)\in \mathcal{C}(G_i)$, $d_{T_i}(y)=0$ and $G_i-E(T_i)$ is an outerplanar diagram ($i=1,2,3$). Meanwhile, we have $\max\{d_{T_1}(y_1),d_{T_1}(y_2^L),d_{T_2}(y_2^L),d_{T_2}(y_2^R), d_{T_3}(y_2^R),d_{T_3}(y_k)\}\leq 1$. Let $T=T_1\cup T_2\cup T_3$. Note that there is no edges whose end points are belong to different parts of the vertex partition $[\mathcal{V}(y,y_2^L),\mathcal{V}(y_2^L,y_2^R),\mathcal{V}(y_2^R,y)]$ (because otherwise either $yy_2$ or $y_2^Ly_2^R$ may be crossed twice). So $T$ is still a forest. Since $d_T(y_2^R)\leq d_{T_2}(y_2^R)+d_{T_3}(y_2^R)\leq 2$ and $d_T(y_2^L)\leq d_{T_1}(y_2^L)+d_{T_2}(y_2^L)\leq 2$, $\Delta(T)\leq 2$. Thus, a linear forest $T$ has been constructed. Since $\mathcal{C}(G_i)\subseteq \mathcal{C}(G)$ ($i=1,2,3$), $E(T)=E(T_1)\cup E(T_2)\cup E(T_3)\in \mathcal{C}(G_1)\cup\mathcal{C}(G_1)\cup\mathcal{C}(G_3)\in \mathcal{C}(G)$. Meanwhile, $d_T(y)=d_{T_1}(y)+d_{T_2}(y)+d_{T_3}(y)=0$, $d_T(x)=d_T(y_k)=d_{T_3}(y_k)\leq 1$ and $d_T(z)=d_{T}(y_1)=d_{T_1}(y_1)\leq 1$. At last since $G-E(T)\subseteq \bigcup_{i=1}^3 (G_i-E(T_i))$, $G-E(T)$ is an outerplanar diagram. Hence we construct a linear forest $T$ as required in $G$ and completes the proof of the theorem.
\end{proof}

A \emph{star forest} is a graph in which every component is a star. The \emph{root} of a star is the vertex of maximum degree. Note that $K_2$ has two roots. The $roots$ of a star forest is the union of the root of each star component. The following Theorem \ref{star} is an analog of Theorem \ref{linear} (note that the condition $\max\{d_T(x),d_T(z)\}\leq 1$ in Theorem \ref{linear} is equivalent to that $x$ or $z$ are vertices of $T$ if and only if $x$ or $z$ are leaves of $T$), whose proof is almost the same with that of Theorem \ref{linear}. Actually, we can still proceed by induction on the order of $G$ and split the proofs into three cases: the first is $d_G(y)=2$, the second is $d_G(y)=3$ and $xz\in E(G)$, and the last is $d_G(y)\geq 3$ and $xz\not\in E(G)$. In each case we can construct a star forest $T$ as required by the same way as in the proof of Theorem \ref{linear}. The detailed proof of Theorem \ref{star} is left to the readers.

\begin{theorem} \label{star}
Let $G$ be a hamiltonian pseudo-outerplanar diagram and $C$ be the hamiltonian boundary of this diagram. Let $y\in V(C)$ and $yx,yz\in E(C)$. Then there exists a star forest $T$ in $G$ such that $E(T)\in \mathcal{C}(G)$, $d_T(y)=0$, $x$ or $z$ are vertices of $T$ if and only if $x$ or $z$ are roots of $T$, and
$G-E(T)$ is an outerplanar diagram.
\end{theorem}

\begin{corollary} \label{L}
Each pseudo-outerplanar graph can be decomposed into an outerplanar graph and a linear forest, or an outerplanar graph and a star forest.
\end{corollary}

\begin{proof}
Without loss of generality, let $G$ be a quasi-hamiltonian pseudo-outerplanar diagram. Otherwise we can add some edges to close the circumferential boundary of each block. In what follows, we proceed by induction on the number of blocks, $\omega(G)$, in $G$. The base case when $\omega(G)=1$ follows directly from Theorems \ref{linear} and \ref{star} so we assume that $\omega(G)\geq 2$. Choose a block $B$ of $G$ that contains only one cut vertex $y$ (i.e. $B$ is an end-block).
By Theorems \ref{linear} and \ref{star}, $B$ can be decomposed into an outerplanar graph $H_1$ and a linear forest $T_1$ with $d_{T_1}(y)=0$, or an outerplanar graph $H_2$ and a star forest $T_2$ with $d_{T_2}(y)=0$. Meanwhile, by the induction hypothesis, $G-B$ can also be decomposed into an outerplanar graph $H_3$ and a linear forest $T_3$, or an outerplanar graph $H_4$ and a star forest $T_4$. Therefore, $G$ can be covered by the linear forest $T=T_1\cup T_3$ and the outerplanar graph $H=H_1\cup H_3$, or the star forest $T=T_2\cup T_4$ and the outerplanar graph $H=H_2\cup H_4$.
\end{proof}

\begin{theorem} \label{mat}
For every integer $n\geq 12$, there exists a $2$-connected pseudo-outerplanar graph with order $n$ that cannot be decomposed into an outerplanar
graph and a matching.
\end{theorem}

\begin{proof}
We show the last graph $G$ in Figure \ref{covering} is a graph that cannot be decomposed into an outerplanar graph and a matching. Otherwise we suppose that $E(G)=E(H)\cup E(M)$, where $H$ is an outerplanar and $M$ is matching. Set $S_i=\{v_iv_{i+1},v_iv_{i+2},v_iv_{i+3},v_{i+1}v_{i+3},v_{i+2}v_{i+3}\}$ (mod $12$) ($i=1,4,7,10$). We then claim that there exists an edge set $S_i$ that is contained in $E(H)$.
Suppose not, assume first that $v_1v_2\in E(M)$. Then $v_1v_k\in E(H)$ ($k=3,4,7,10,11,12$) and exactly one of $v_{10}v_{11}$ and $v_{10}v_{12}$ should be contained in $E(M)$, say $v_{10}v_{11}$. Then $v_kv_{10}\in E(H)$ ($k=4,7,12$). However, the five vertices $\{v_1,v_4,v_7,v_{10},v_{12}\}$ and the three disjoint paths $\{v_1v_4v_{10},v_1v_7v_{10},v_1v_{12}v_{10}\}$ form a copy of $K_{2,3}$ in $H$; this is a contradiction. Now assume that $v_1v_4\in E(M)$. Then $v_1v_2,v_1v_3,v_1v_7,v_2v_4,v_3v_4,v_4v_7\in E(H)$ and thus the graph induced by $\{v_1,v_2,v_3,v_{4},v_{7}\}$ is a $K_{2,3}$, which is impossible in an outerplanar graph. Hence in the following we assume that $S_1\subseteq E(H)$. If $\{v_1v_7,v_4v_7\} \subseteq E(H)$, then the five vertices $\{v_1,v_2,v_3,v_{4},v_{7}\}$ and the three disjoint paths $\{v_1v_2v_4,v_1v_3v_4,v_1v_{7}v_4\}$ form a copy of $K_{2,3}$ in $H$, a contradiction. So exactly one of $v_1v_7$ and $v_4v_7$ should be contained in $E(M)$, say $v_1v_7$. Similarly, $\{v_1v_{10},v_4v_{10}\} \not\subseteq E(H)$. Thus $v_1v_{10}\in E(H)$, $v_4v_{10}\in E(M)$ and $v_7v_{10}\in E(H)$. Now the six vertices $\{v_1,v_2,v_3,v_4,v_7,v_{10}\}$ and the three disjoint paths $\{v_1v_3v_4,v_1v_2v_4,v_1v_{10}v_7v_4\}$ form a $K_{2,3}$-minor in $H$. This contradiction completes the proof of this theorem.
\end{proof}

\begin{theorem} \label{maximal}
Every maximal pseudo-outerplanar graph $G$ is obtained from a maximal pseudo-outerplanar diagram $H$ by gluing a $K_3$ or a $K_4$ along a boundary edge of $H$.
\end{theorem}

\begin{proof}
Without loss of generality, we assume that $G$ is a 2-connected maximal pseudo-outerplanar diagram. Since $G$ is maximal, $G$ is hamiltonian and $G$ has at least one chord. Let $C=\{v_1v_2\cdots v_{|G|}\}$ be the hamiltonian boundary of the diagram of $G$. Now we split the proof into two cases.

\noindent\emph{Case $1$. There exists a crossed chord in $G$.}

Let $v_iv_j$ be a chord in $G$ that crosses another chord $v_kv_l~(1\leq i<k<j<l\leq |G|)$. Actually, we can properly choose $i$ and $j$ such that there is no pair of mutually crossed chords in $\mathcal{C}[v_i,v_l]\setminus \{v_iv_j,v_kv_l\}$, because otherwise we can change the value of $i$ or $j$ to meet this condition.

Assume first that there is no non-crossed chord in $\mathcal{C}[v_i,v_l]\setminus \{v_iv_l\}$. Then we shall have $k=i+1$. Otherwise, since $v_iv_k\not\in E(G)$ by our assumption, we can add $v_iv_k$ to $G$ so that $G$ is still pseudo-outerplanar, contradicting the fact that $G$ is maximal. Similarly, $j=k+1,~l=j+1$ and $v_{i}v_{l}\in E(G)$ by the maximality of $G$. Furthermore, $d(v_k)=d(v_j)=3$. Now remove the vertices $v_k$ and $v_j$ from $G$ and denote the resulting graph by $H$. Then $H$ is a maximal pseudo-outerplanar diagram. Otherwise we can add an edge $e=v_av_b\not\in E(H)~(a,b\neq k~{\rm or}~j)$ to $H$ so that $H+e$ is pseudo-outerplanar. Therefore, $e\not\in E(G)$ and $G+e$ is a pseudo-outerplanar graph, contradicting the fact that $G$ is maximal. At this stage, one can easily see that $G$ is obtained from $H$ by gluing a $K_4$ along the boundary edge $v_iv_l$ of $H$.

Second, assume that there is a non-crossed chord $v_rv_s$ in $\mathcal{C}[v_i,v_l]\setminus \{v_iv_l\}$. Since there is no crossed chords in $\mathcal{C}[v_r,v_s]$ by assumption, we can properly choose $r$ and $s$ such that $\mathcal{C}[v_r,v_s]\setminus \{v_rv_s\}=\emptyset$. By the maximality of $G$, we have $s=r+2$, otherwise we can add an edge $v_rv_{r+2}$ to $G$ so that the resulting graph is still pseudo-outerplanar, a contradiction. Since $v_rv_s$ is a non-crossed chord, $d(v_{r+1})=2$. Now remove the vertex $v_{r+1}$ from $G$ and denote the resulting graph by $H'$. Then by a similar argument as before one can prove that $H'$ is a maximal pseudo-outerplanar diagram. Furthermore, one can easily see that $G$ is obtained from $H'$ by gluing a $K_3$ along the boundary edge $v_rv_{r+2}$ of $H$.

\noindent\emph{Case $2$. There exists a non-crossed chord in $G$.}

Let $v_iv_j~(1\leq i<j\leq |G|)$ be a non-crossed chord in $G$. In this case we shall assume that there is no crossed chord in $\mathcal{C}[v_i,v_j]$, because otherwise we are in Case 1. We can also properly choose $i$ and $j$ such that $\mathcal{C}[v_i,v_j]\setminus \{v_iv_j\}=\emptyset$. Therefore, we are now in the second subcase of Case 1, where we can set $r:=i$ and $s:=j$.
\end{proof}

\begin{corollary} \label{mat}
Each pseudo-outerplanar graph can be decomposed into two forests and a matching.
\end{corollary}

\begin{proof}
Let $G$ be a pseudo-outerplanar graph. In the following, we proceed by induction on the size of $G$ and assume that $G$ is a maximal pseudo-outerplanar diagram. By Theorem \ref{maximal}, there respectively exists a $K_3=[xyz]$ or a $K_4=[xyuv]$
contained in $G$ such that $H=G-\{xz,yz\}$ or $H=G-\{xu,xv,yu,yv,uv\}$ is a maximal pseudo-outerplanar graph with $xy$ being its boundary edge.
By induction on $H$, there exists two forests $F_{1}$, $F_{2}$ and a matching $M$ such that $E(H)=E(F_{1})\cup E(F_{2})\cup E(M)$.
In the former case, let $F'_{1}=F_{1}+xz$, $F'_{2}=F_{2}+yz$ and $M'=M$; and in the latter case, let $F'_1=F_1+\{xu,xv\}$, $F'_2=F_2+\{yu,yv\}$ and $M'=M+uv$.
One can easily check that the two forests $F'_{1}$, $F'_{2}$ and the matching $M'$ are the desired decomposition of $G$.
\end{proof}

\begin{theorem}\label{2f}
For every integer $n\geq 6$, there exists a $2$-connected pseudo-outerplanar graph with order $n$ that cannot be decomposed into two forests.
\end{theorem}

\begin{proof}
Let $C=v_1\cdots v_n v_1$($n\geq 6$) be a cycle with $n$ vertices. We add edges $v_1v_i$ for all $3\leq i\leq n-1$ and edges $v_{2i}v_{2i+2}$ for all $1\leq i\leq \lfloor\frac{n}{2}\rfloor-1$. One can easily check that the resulted graph $G_n$ is a $2$-connected pseudo-outerplanar graph with order $n$ and size $\lfloor\frac{5}{2}n\rfloor-4$. If $G_n$ can be decomposed into two forests $F_1$ and $F_2$, then $|E(G_n)|=|E(F_1)|+|E(F_2)|\leq |V(F_1)|+|V(F_2)|-2\leq 2n-2$. However, for $n\geq 6$, $|E(G_n)|=\lfloor\frac{5}{2}n\rfloor-4>2n-2$. Hence, the graph $G_n$($n\geq 6$) cannot be covered by two forests.
\end{proof}

From Corollary \ref{mat} and Theorem \ref{2f}, we directly have the following two corollaries.

\begin{corollary}
Every pseudo-outerplanar graph is $(2,1)$-coverable; the two parameters given here are best possible.
\end{corollary}

\begin{corollary}
The arboricity of a pseudo-outerplanar graph is at most $3$; and this bound is sharp.
\end{corollary}

\section{Unavoidable Structures}
In this section, a vertex set $\mathcal{V}[v_i,v_j]~(i<j)$ is called a \emph{non-edge} if $j=i+1$ and $v_iv_j\not\in E(G)$, called a \emph{path} if $v_k v_{k+1}\in E(G)$ for all $i\leq k<j$ and called a \emph{subpath} if $j>i+1$ and some edges in the form $v_kv_{k+1}~(i\leq k<j)$ are missing. We say a chord $v_kv_l~(k<l)$ is contained in a chord $v_iv_j~(i<j)$ if $i\leq k$ and $l\leq j$. In any figure of this section, the solid vertices have no edges of
$G$ incident with them other than those shown.

\begin{lemma} {\rm \cite{Wang}}\label{strc_opg}
Let $G$ be a $2$-connected outerplanar graph. Then

$(1)$ $G$ has two adjacent $2$-vertices $u$ and $v$, or

$(2)$ $G$ has a $3$-cycle $uwxu$ such that $d(u)=2$ and $d(w)=3$, or

$(3)$ $G$ has a $4$-vertex $w$, where $N(w)=\{u,v,x,y\}$, such that
$d(u)=d(v)=2$, $N(u)=\{w,x\}$ and $N(v)=\{w,y\}$.
\end{lemma}

For the class of pseudo-outerplanar graphs, we have a similar structural theorem as Lemma \ref{strc_opg}. But it seems much more complex since crossings are permitted in a pseudo-outerplanar graph.

\begin{theorem} \label{2con}
Let $G$ be a pseudo-outerplanar diagram with $\delta(G)\geq 2$. Then $G$ contains one of the following configurations $G_1$--$G_{17}$. Moreover,\\
\indent {\rm (a)} if $G$ contains some configuration among $G_6$--$G_{17}$, then the drawing of this configuration in the figure is a part of the diagram of $G$ with its bending edges corresponding to the chords;\\
\indent {\rm (b)} if $G$ contains the configuration $G_3$ and $xy\not\in E(G)$, where $x$ and $y$ are the vertices of $G_3$ as described in the figure, then we can properly add an edge $xy$ to $G$ so that the resulting diagram is still pseudo-outerplanar.
\begin{center}
  \includegraphics[width=15.0cm,height=6cm]{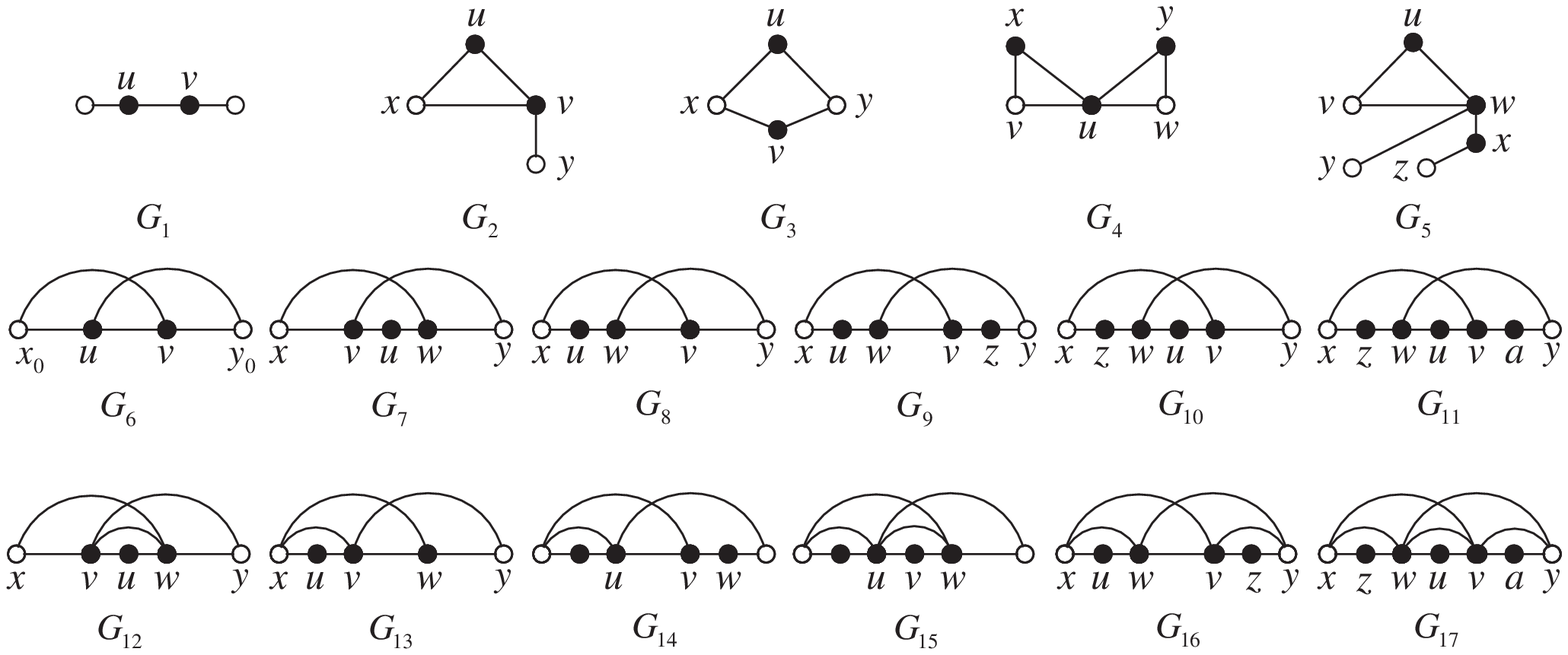}\\
  %\caption{Configurations of $2$-connected pseudo-outerplanar graphs}
  \label{config}
\end{center}
\end{theorem}

\begin{proof}
We first consider the case when $G$ is a $2$-connected pseudo-outerplanar diagram. Recall that this diagram minimizes the number of crossings. Let $v_1,v_2,\cdots,v_{|G|}$ be the vertices of this diagram lying in a clockwise sequence. If there is no crossings in $G$, then $G$ is an outerplanar graph and thus $G$ satisfies this theorem by Lemma \ref{strc_opg}. Otherwise, we can properly choose one chord $v_iv_j$ such that \vspace{1mm}

(1) $v_iv_j$ crosses $v_kv_l$ in $G$;

(2) $v_i,v_k,v_j$ and $v_l$ are lying in a clockwise sequence;

(3) besides $v_iv_j$ and $v_kv_l$, there is no crossed chords in $\mathcal{C}[v_i,v_l]$.\vspace{1mm}

The condition (3) can be easily fulfilled, because otherwise we could change the values of $i$ and $j$ to meet this condition (note that the values of $k$ and $l$ are determined by $i$ and $j$). Without loss of generality, assume that $1\leq i<k<j<l\leq |G|$, because otherwise we can adjust the labellings of the vertices in $G$ to meet it.

\vspace{1mm}\noindent \emph{Claim $1$. $\mathcal{V}[v_i,v_k]$ is
either non-edge or path, and so do $\mathcal{V}[v_k,v_j]$
and $\mathcal{V}[v_j,v_l]$.}

\vspace{1mm}We only need to prove that $\mathcal{V}[v_i,v_k]$ cannot be subpath. Otherwise there exists two vertices $v_m$ and $v_{m+1}$, where $i\leq m\leq k-1$, such that $v_mv_{m+1}\not\in E(G)$. If there are chords in the form $v_av_{m+1}$ such that $i\leq a\leq m-1$, then we choose one among them such that $a$ is maximum. One can see that $v_a$ is a vertex cut of $G$, because there is no edges between $\mathcal{V}[v_{a+1},v_m]$ and $\mathcal{V}[v_{m+1},v_{a-1}]$ by the choice of $a$ and (3). This contradicts the fact that $G$ is 2-connected. Thus there is no chords in the form $v_av_{m+1}$ such that $i\leq a\leq m-1$. Similarly, there is no chords in the form $v_mv_{b}$ such that $m+2\leq b\leq k$.
Let $p=\max\{n|v_{m+1}v_n\in E(G),m+1<n\leq k\}$ and $q=\min\{n|v_nv_m\in E(G),i\leq n<m\}$. Since $\mathcal{V}[v_i,v_k]$ is neither non-edge nor path, we have $k-i\geq 2$ and thus at least one of the integers $p$ and $q$ exists. Without loss of generality suppose that $p$ exists. Then $v_p$ is a vertex cut of $G$, because there is no edges between $\mathcal{V}[v_{m+1},v_{p-1}]$ and $\mathcal{V}[v_{p+1},v_{m}]$ by the choices of $m$, $p$ and by (3). This contradiction completes the proof of Claim 1.

\vspace{1mm}\noindent \emph{Claim $2$. If $\mathcal{V}[v_i,v_k]$ is a path and $k-i\geq 3$, then $G$ has a subgraph isomorphic to one of the configurations $\{G_1,G_2,G_4\}$. This result also holds for $\mathcal{V}[v_k,v_j]$ and $\mathcal{V}[v_j,v_l]$ if
$j-k\geq 3$ and $l-j\geq 3$, respectively.}

\vspace{1mm} Suppose that there is no other chord except $v_iv_k$ (if exists) in $\mathcal{V}[v_i,v_k]$, then the configuration $G_1$ occurs, since $k-i\geq 3$. So we assume that $S:=\mathcal{C}[v_i,v_k]\setminus \{v_iv_k\}\neq \emptyset$. Now we prove that there exists at least one chord in $S$ that contains at least one other chord. Suppose that such a chord does not exist. Then we first choose a chord $v_mv_n\in S ~(m<n)$. Without loss of generality, assume that $n\neq k$. If $n-m\geq 3$, then we can easily find a copy of $G_1$ in $G$, since $v_mv_n$ contains no other chords by our assumption. If $n-m=2$, then it is trivial to see that $d(v_{m+1})=2$. Now if $\min\{d(v_m),d(v_n)\}\leq 3$, then a copy of $G_2$ would be found. Thus we shall assume that $\min\{d(v_m),d(v_n)\}\geq 4$. So there exists another chord $v_nv_p~(n<p)$ in $S$, since $d(v_n)\geq 4$ and $v_mv_n$ cannot be contained in a chord in the form $v_qv_n~(q<n)$ by the assumption. Similarly, we shall assume that $p-n=2$ and $d(v_{n+1})=2$ for otherwise the configuration $G_1$ would be found. Now one can see that $d(v_n)=4$, because otherwise there would be chord in $S$ that contains either $v_mv_n$ or $v_nv_p$, a contradiction. Therefore, the graph induced by $\mathcal{V}[v_m,v_p]$ contains the configuration $G_4$. Thus we can choose one chord $v_av_b\in S~(a<b)$ such that $v_av_b$ contains at least one chord, and furthermore, every chord contained in $v_av_b$ contains no other chords (this condition can be easily fulfilled by properly changing the values of $a$ and $b$ if necessary). Let $v_mv_n~(m<n)$ be the chord contained in $v_av_b$. Then by the similar argument as above, we have to consider the case when $n-m=2$, $d(v_{m+1})=2$ and $\min\{d(v_m),d(v_n)\}\geq 4$. Without loss of generality, assume that $n\neq b$. Then there must be a chord $v_nv_p~(n<p\leq b)$ in $S$, since $d(v_n)\geq 4$ and $v_mv_n$ can not be contained in a chord in the form $v_qv_n~(q<n)$ by the choices of $a$ and $b$. By the similar argument as before, if $G$ contains no copies of $G_1$ or $G_2$, then $p-n=2$ and $d(v_{n+1})=2$. Furthermore, one can similarly prove that $d(v_n)=4$ by the choices of $a$ and $b$.
Thus we would find a copy of $G_4$ in the graph induced by $\mathcal{V}[v_m,v_p]$.

\vspace{1mm}\noindent \emph{Claim $3$. At most one of
$\mathcal{V}[v_i,v_k]$, $\mathcal{V}[v_k,v_j]$ and
$\mathcal{V}[v_j,v_l]$ can be non-edge.}

\vspace{1mm} If $\mathcal{V}[v_i,v_k]$ and $\mathcal{V}[v_k,v_j]$ are non-edge, then it is trivial that $v_l$ is a vertex cut of $G$, contradicting the fact that $G$ is $2$-connected. If $\mathcal{V}[v_i,v_k]$ and $\mathcal{V}[v_j,v_l]$ are non-edge, then we can adjust the drawing of $G$ by replacing the vertices order $\{v_i,v_k,v_{k+1},\cdots,v_{j-1},v_j,v_l\}$ with $\{v_i,v_j,v_{j-1},\cdots,v_{k+1},v_k,v_l\}$. This operation can reduce the number of crossings in the drawing of $G$ by one, contradicting the
assumption that this diagram minimizes the number of crossings.

\vspace{1mm}\noindent \emph{Claim $4$. If one of
$\mathcal{V}[v_i,v_k]$, $\mathcal{V}[v_k,v_j]$ and
$\mathcal{V}[v_j,v_l]$ is non-edge, then $G$ has a subgraph
isomorphic to one of the configurations $\{G_1,G_2,G_3\}$.}

\vspace{1mm}Suppose that $\mathcal{V}[v_i,v_k]$ is a non-edge. By Claims 1--3, both $\mathcal{V}[v_k,v_j]$ and $\mathcal{V}[v_j,v_l]$ are paths with $1\leq j-k\leq 2$ and $1\leq l-j\leq 2$. If $j-k=2$ and $v_kv_j\in E(G)$, then it is clear that $d(v_k)=3$ and $d(v_{k+1})=2$, implying that the configuration $G_2$ occurs. If $j-k=2$ but $v_kv_j\not\in E(G)$, then $d(v_k)=d(v_{k+1})=2$, implying that the configuration $G_1$ occurs. So we assume that $j=k+1$. If $l=j+2$, then $d(v_{j+1})=2$ whenever $v_jv_l$ is an chord or not. In this case the configuration $G_3$ occurs since $d(v_k)=2$, and moveover, $G+v_jv_l$ is still pseudo-outerplanar if $v_jv_l\not\in E(G)$.
So we assume that $l=j+1$. Now $v_k,v_j,v_l$ form a triangle satisfying $d(v_k)=2$ and $d(v_j)=3$. So the configuration $G_2$ occurs. The case when $\mathcal{V}[v_j,v_l]$ is
a non-edge can be dealt with similarly.

Now suppose that $\mathcal{V}[v_k,v_j]$ is a non-edge. By Claims 1--3, both $\mathcal{V}[v_i,v_k]$ and $\mathcal{V}[v_j,v_l]$ are paths with $1\leq k-i\leq 2$ and $1\leq l-j\leq 2$. If $k-i=2$ or $j-l=2$, by the similar argument as before, we either have $d(v_{k-1})=d(v_k)=2$ or have $d(v_j)=d(v_{j+1})=2$, implying that the configuration $G_1$ occurs. So we assume that $k-i=l-j=1$. In this case the four vertices $v_i,v_j,v_l$ and $v_k$ form a quadrilateral with $d(v_i)=d(v_k)=2$, which implies that the configuration $G_3$ occurs in $G$ and furthermore, $G+v_iv_l$ is still pseudo-outerplanar if $v_iv_l\not\in E(G)$.

\vspace{1mm}In the following, we assume that $\mathcal{V}[v_i,v_k]$, $\mathcal{V}[v_k,v_j]$ and $\mathcal{V}[v_j,v_l]$ are all paths, where $\max\{k-i,j-k,l-j\}\leq 2$. Set $X=\mathcal{C}[v_i,v_l]\backslash \{v_iv_j,v_kv_l\}$ and $x=|X|$. It
is clear that $x\leq 3$.

\vspace{1mm}\noindent \emph{Claim $5$. If $x=0$, then $G$ has a subgraph isomorphic to one of the configurations $G_6$--$G_{11}$; If $x=1$, then $G$ has a subgraph isomorphic one of the configurations $\{G_5,G_{12},G_{13},G_{14}\}$; If $x=2$, then $G$ has a subgraph isomorphic to one of the configurations $\{G_5, G_{15},G_{16}\}$; If $x=3$, then $G$ has a subgraph
isomorphic to the configuration $G_{17}$.}

\vspace{1mm}Here, we just show the case when $x=2$ and $v_kv_j,v_jv_l\in X$ for example, and leave the discussions about other cases to the readers since they are quite similar. In fact, if $k-i=1$ (resp. $k-i=2$), then the configuration $G_{15}$ (resp. $G_5$) would occurs in $G$ since $d(v_k)=4$ and $d(v_{i+1})=d(v_{k+1})=d(v_{j+1})=2$, and furthermore the drawing of the configuration $G_{15}$ (resp. $G_5$) in the figure is just a part of the diagram of $G$ with its bending edges corresponding to the chords.

\vspace{1mm} Until now, Claims 1-5 just complete the proof of this theorem for
the case when $G$ is $2$-connected. Now we suppose that $G$ has at least two blocks. Let $B$ be an end block and let $v_1,v_2,\cdots,v_{|B|}$ be the vertices of $B$ that lies in a clockwise sequence. Without loss of generality, let $v_1$ be the unique cut vertex of $B$.

\vspace{1mm}\noindent \emph{Claim $6$. $B$ is an outerplanar graph.}

\vspace{1mm} We prove that there is no crossings in $B$. Suppose, to the contrary, that there is a chord $v_iv_j$ that crosses another chord $v_kv_l$, where $1\leq i<k<j<l$. Note that the chord $v_iv_j$ satisfies (1) and (2) now. If it does not fulfill (3) at this stage. Then there must be at least one pair of mutually crossed chords contained in either $\mathcal{C}[v_i,v_k]$, or $\mathcal{C}[v_k,v_j]$, or $\mathcal{C}[v_j,v_l]$. We choose one pair $v_av_b$ and $v_cv_d$ among them such that $a<c<b<d$ and there is no other crossed chords in $\mathcal{C}[v_a,v_d]$ besides $v_av_b$ and $v_cv_d$. Now set $i:=a$, $j:=b$, $k:=c$ and $l:=d$. Therefore, in any case we can find a pair of mutually crossed chords, $v_iv_j$ and $v_kv_l$, such that $1\leq i<k<j<l$ and the three conditions at the beginning of the proof are fulfilled. Note that $B$ is an 2-connected pseudo-outerplanar diagram. Thus we can set $v_i$, $v_j$, $v_k$, $v_l$ as we did in the 2-connected case. Recall the proofs of Claims 1-5, every time we find a copy of some configuration the vertices $v_i$ and $v_l$ cannot be the solid vertices (i.e. the degrees of them in the configuration shall not necessarily to be confirmed). For a vertex $v\in V(B)\setminus \{v_1\}$, its degree in $B$ is equal to its degree in $G$, since $B$ is an end block and $v_1$ is the unique cut vertex of the $B$. Among the vertices in $\mathcal{V}[v_i,v_l]$, only $v_i$ may be the cut vertex since $1\leq i<k<j<l$. Therefore, the proofs of Claims 1-5 are also valid for this claim and then the same results would be obtained.

\vspace{1mm}\noindent \emph{Claim $7$. $B$ has a subgraph isomorphic
to one of the configurations $\{G_1,G_2,G_4\}$ in such a way that $v_1$
is not a solid vertex.}

\vspace{1mm} Since $B$ is a $2$-connected outerplanar graph, $B$ is hamiltonian. So $\mathcal{V}[v_1,v_{|B|}]$ is a path. The proof of Claim 2 implies that if $\mathcal{V}[v_i,v_k]$ is a path with $k-i\geq 3$ such that there is no crossed edges in $\mathcal{C}[v_i,v_k]$ and no edges between $\mathcal{V}(v_i,v_k)$ and $\mathcal{V}(v_k,v_i)$, then $G$ contains one of $\{G_1,G_2,G_4\}$ in such a way that $v_i$ and $v_k$ are not the solid vertices. Thus in this claim, if $|B|\geq 4$, then we set $i:=1$, $k:=|B|$ and come back to the proof of Claim 2. If $|B|\leq 3$, then it is trivial to see that $G_1$ would appear. This contradiction completes the proof of the theorem for
the case when $G$ has cut vertices.
\end{proof}

The following is a straightforward corollary of Theorem \ref{2con}.

\begin{corollary}\label{mindegree3}
Each pseudo-outerplanar graph contains a vertex of degree at most $3$.
\end{corollary}

\section{Edge Chromatic Number and Linear Arboricity}

In this section, we aim to consider the problems of covering a pseudo-outerplanar graph $G$ with $\Delta(G)$ matchings or $\lceil\frac{\Delta(G)}{2}\rceil$ linear forests. A graph $G$ is \emph{$\chi'$-critical} if $\chi'(G)=\Delta(G)+1$ but $\chi'(H)\leq \Delta(G)$ for any proper subgraph $H\subset G$, is \emph{la-critical} if $la(G)>\lceil\frac{\Delta(G)}{2}\rceil$ but $la(H)\leq \lceil\frac{\Delta(G)}{2}\rceil$ for any proper subgraph $H\subset G$.

%In the proof of Theorems \ref{edge chromatic number} and \ref{linear arborocity}, we consider a minimal counterexample $G$ to the theorem (hence $G$ is critical). Obviously $\delta(G)\geq 2$, so $G$ contains at least one of the configurations $G_1$--$G_{17}$ by Theorem \ref{2con}.

%Then we delete some vertices from some configuration that $G$ contains and therefore get a new graph $G'$ with smaller size than $G$. Following the proof of Theorem \ref{strc_opg}, one can easily prove the pseudo-outerplanarity of $G'$ in each case. So by induction, we can color $G'$ such that the subgraph induced by each color set is a matching or a linear forest. Then we construct a coloring $\col$ of $G$ from the coloring $\phi$ of $G'$.

\begin{lemma} \label{ed}
If $G$ is $\chi'$-critical and $uv\in E(G)$, then $d(u)+d(v)\geq \Delta(G)+2$.
\end{lemma}

\begin{lemma} \label{la}
If $G$ is la-critical and $uv\in E(G)$, then $d(u)+d(v)\geq 2\lceil\frac{\Delta(G)}{2}\rceil+2$.
\end{lemma}

The above two lemmas are very classic and useful; their proofs can be found in \cite{Bondy} and \cite{Wu} respectively.
Given a coloring $\varphi$ of $G$, $c_j(v)$ denotes the number of edges incident with $v$ colored by $j$. Let $C_{\varphi}^{i}(v)=\{j|c_j(v)=i\}$, $i=0,1,2$.
Then $C_{\varphi}^{0}(v)\cup C_{\varphi}^{1}(v)=\{1,2,\cdots,k\}$ if $\varphi$ is a proper $k$-edge-coloring, and $C_{\varphi}^{0}(v)\cup C_{\varphi}^{1}(v)\cup C_{\varphi}^{2}(v)=\{1,2,\cdots,k\}$ if $\varphi$ is a $k$-tree-coloring. For brevity, in the proof of Theorem \ref{edge chromatic number} we use the notion \emph{$k$-coloring} to replace the statements of proper $k$-edge-coloring or $k$-tree-coloring and use the notion \emph{PO-graph} to replace the statement of pseudo-outerplanar graph. For a graph $G$ and two distinct vertices $u,v\in V(G)$, denote by $G+xy$ the graph obtained from $G$ by adding an new edge $xy$ if $xy\not\in E(G)$, or $G$ itself if $xy\in E(G)$.

\begin{theorem} \label{edge chromatic number}
Let $G$ be a pseudo-outerplanar graph. If $\Delta(G)\geq 4$, then $\chi'(G)=\Delta(G)$.
\end{theorem}

\begin{proof}
Suppose for a contradiction that there exists a minimal (in terms of the size) pseudo-outerplanar diagram $G$ with $\Delta(G)\geq 4$ that has no $\Delta(G)$-coloring. One can easily observe that $G$ is $2$-connected and $\chi'$-critical. By Theorem \ref{2con} and Lemma \ref{ed}, $G$ contains at least one of the configurations $\{G_3,G_4,G_5,G_6,G_{12},G_{13},G_{16},G_{17}\}$. Set $S=\{1,2,\cdots,\Delta(G)\}$.

If $G\supseteq G_3$, then the pseudo-outerplanar graph $G'=G\backslash \{u,v\}$ admits a $\Delta(G)$-coloring $\phi$ by induction hypothesis (when $\Delta(G')=\Delta(G)$) or Vizing's Theorem (when $\Delta(G')\leq \Delta(G)-1$). Construct a $\Delta(G)$-coloring $\col$ of $G$ as follows. If $C^1_{\phi}(x)=C^1_{\phi}(y):=L$ (notice that $|L|= \Delta(G)-2$ by Lemma \ref{ed}), then let $\col(ux)=\col(yv)\in S\setminus L$ and $\col(uy)=\col(xv)\in S\setminus (L\cup \{\col(ux)\})$. If $C^1_{\phi}(x)\neq C^1_{\phi}(y)$, then $(S\setminus C^1_{\phi}(x))\cap C^1_{\phi}(y)\neq \emptyset$ since $d(x)=d(y)=\Delta(G)$ by Lemma \ref{ed}. Let $\col(ux)\in (S\setminus C^1_{\phi}(x))\cap C^1_{\phi}(y)$, $\col(xv)\in S\setminus (C^1_{\phi}(x)\cup \{\col(ux)\})$, $\col(vy)\in S\setminus (C^1_{\phi}(y)\cup \{\col(xv)\})$ and $\col(uy)\in S\setminus (C^1_{\phi}(y)\cup \{\col(yv)\})$.
In each case, we color the remain edges of $G$ by the same colors used in $\phi$. Thus, we have constructed a $\Delta(G)$-coloring $\col$ of $G$ from the $\Delta(G)$-coloring $\phi$ of $G'$. In the next cases, while constructing a coloring $\col$ of $G$ from the coloring $\phi$ of $G'$, we only give the colorings for the edges in $E(G)\setminus E(G')$, since for every edge $e\in E(G)\cap E(G')$ we always let $\col(e)=\phi(e)$.

If $G\supseteq G_4$, we shall assume that $d(v)=d(w)=\Delta(G)=4$ because of Lemma \ref{ed}. Then the PO-graph $G'=G\setminus\{x,y,u\}$ admits a 4-coloring $\phi$. Construct a 4-coloring $\col$ of $G$ as follows, where two cases are considered without loss of generality (wlog. for short).
If $C^1_{\phi}(v)=C^1_{\phi}(w)=\{1,2\}$, then let $\col(uy)=1$, $\col(ux)=2$, $\col(uw)=\col(vx)=3$ and $\col(uv)=\col(wy)=4$. If $C^1_{\phi}(v)=\{1,2\}$, $1\not\in C^1_{\phi}(w)$ and $3\in C^1_{\phi}(w)$, then let $\col(uw)=1$, $\col(ux)=2$, $\col(xv)=\col(uy)=3$, $\col(uv)=4$ and $\col(wy)\in \{2,3,4\}\setminus C^1_{\phi}(w)$.

If $G\supseteq G_5$, we shall assume that $d(v)=\Delta(G)=4$ because of Lemma \ref{ed}. Then the PO-graph $G'=G\setminus\{u\}$ admits a 4-coloring $\phi$.
One can easily see that $(C^1_{\phi}(v)\cap C^1_{\phi}(w))\setminus \{\phi(vw)\}\neq \emptyset$, because otherwise $vw$ would be incident with four colors under $\phi$. Assume that $C^1_{\phi}(v)=\{1,2,3\}$ and $\phi(vw)=3$ wlog. If $C^1_{\phi}(w)\neq C^1_{\phi}(v)$, then assume that $C^1_{\phi}(w)=\{1,3,4\}$ wlog. Whereafter, we can extend $\phi$ to a 4-coloring of $\col$ of $G$ by taking $\col(uv)=4$ and $\col(uw)=2$.
If $C^1_{\phi}(w)=C^1_{\phi}(v)$, then we consider two subcases. If $\phi(xz)=4$, then construct a 4-coloring of $G$ by recoloring $wx$ and $wv$ with 3 and 4, and coloring $uv$ and $uw$ with 3 and 2, respectively. If $\phi(xz)\neq 4$, then construct a 4-coloring of $G$ by recoloring $wx$ with 4 and coloring $uv$ and $uw$ with 4 and 2, respectively.

If $G\supseteq G_6$, we shall assume that $\min\{d(x_0),d(y_0)\}\geq 3$ and $\Delta(G)=4$ by Lemma \ref{ed}. Assume first that $d(x_0)=d(y_0)=4$. If $x_0y_0\not\in E(G)$, then let $N(x_0)=\{u,v,x_1,x_2\}$ and $N(y_0)=\{u,v,y_1,y_2\}$. Let $G'=G\setminus\{u,v\}+x_0y_0$. By Lemma \ref{2con}, the configuration $G_6$ is a part of the pseudo-outerplanar diagram of $G$. Thus $G'$ can also be a PO-graph and thus $G'$ admits a 4-coloring $\phi$ by the minimality of $G$. Set $M=\{\phi(x_0x_1),\phi(x_0x_2),\phi(y_0y_1),\phi(y_0y_2)\}$ and $m=|M|$. Since the colors used in $\phi$ is at most four and $x_0y_0\in E(G')$, $m\leq 3$ (otherwise the edge $x_0y_0$ cannot be colored under $\phi$ because it is already incident with four colored edges). If $m=3$, assume that $\phi(x_0x_1)=\phi(y_0y_1)=1$, $\phi(x_0x_2)=2$ and $\phi(y_0y_2)=3$ wlog. Now we can extend $\phi$ to a $4$-coloring $\varphi$ of $G$ by taking $\varphi(uv)=1$, $\varphi(vy_0)=2$, $\varphi(ux_0)=3$ and $\varphi(vx_0)=\varphi(uy_0)=4$. If $m\leq 2$, assume that $\phi(x_0x_1)=\phi(y_0y_1)=1$ and $\phi(x_0x_2)=\phi(y_0y_2)=2$ wlog. Now we can also extend $\phi$ to a $4$-coloring $\varphi$ of $G$ by taking $\varphi(uv)=1$, $\varphi(vy_0)=\varphi(ux_0)=3$ and $\varphi(vx_0)=\varphi(uy_0)=4$. On the other hand, if $x_0y_0\in E(G)$, let $N(x_0)=\{u,v,y_0,x_1\}$ and $N(y)=\{u,v,x_0,y_1\}$. Then $x_1\not=y_1$, otherwise by the 2-connectivity of $G$ we have $G\simeq G[\{u,v,x_0,y_0,x_1\}]$, which can be 4-colorable. Consider the graph $G'=G\setminus\{u,v\}-x_0y_0$, which admits a 4-coloring $\phi$ by the minimality of $G$. If $\phi(x_0x_1)=\phi(y_0y_1)=1$, then let $\col(uv)=1$, $\col(x_0y_0)=2$, $\col(ux_0)=\col(vy_0)=3$ and $\col(vx_0)=\col(uy_0)=4$. If $\phi(x_0x_1)=1$ and $\phi(y_0y_1)=2$, then let $\col(vy_0)=1$, $\col(ux_0)=2$, $\col(uv)=\col(x_0y_0)=3$ and $\col(vx_0)=\col(uy_0)=4$. Second, assume that one of $x_0$ and $y_0$ has degree three. Assume that $d(x_0)=3$ wlog. Let $N(x_0)=\{u,v,w\}$. Consider the PO-graph $G'=G-ux_0$. By the minimality of $G$, $G'$ has a 4-coloring $\phi$. If $A:=S\setminus \{\phi(vx_0),\phi(wx_0),\phi(uv),\phi(uy_0)\}\neq \emptyset$ (recall that $S=\{1,2,3,4\}$), then let $\col(ux_0)\in A$. Otherwise, assume that $\phi(vx_0)=1$, $\phi(wx_0)=2$, $\phi(uv)=3$ and $\phi(uy_0)=4$ wlog. Since $d(v)=3$, $\phi(uy_0)=4$ and $vy_0\in E(G')$, $v$ is not incident with the color $4$ under $\phi$. Thus we can extend $\phi$ to a 4-coloring of $G$ by recoloring $vx_0$ with $4$ and then coloring $ux_0$ with 1.

If $G\supseteq G_{12}$, we shall assume that $\Delta(G)=4$ because of Lemma \ref{ed}. Assume first that $d(x)=d(y)=4$. If $xy\not\in E(G)$, then denote $N(x)=\{v,w,x_1,x_2\}$ and $N(y)=\{v,w,y_1,y_2\}$. Consider the graph $G'=G\setminus \{v,w\}+xy+ux+uy$. Since the configuration $G_{12}$ is a part of the pseudo-outerplanar diagram of $G$ by Lemma \ref{2con}, we can properly add three edges $xy$, $ux$ and $uy$ to $G\setminus \{v,w\}$ such that $G'$ is still a PO-graph. Thus
$G'$ admits a 4-coloring $\phi$ by the minimality of $G$. One can see that $\{\phi(xx_1),\phi(xx_2)\}\neq \{\phi(yy_1),\phi(yy_2)\}$ (otherwise we cannot properly color the triangle $uxy$ under $\phi$) and $\{\phi(xx_1),\phi(xx_2)\}\cap \{\phi(yy_1),\phi(yy_2)\}\neq \emptyset$ (otherwise we cannot color the edge $xy$ under $\phi$). Assume that $\phi(xx_1)=1$, $\phi(xx_2)=\phi(yy_1)=2$ and $\phi(yy_2)=3$ wlog. Then we can construct a 4-coloring $\col$ of $G$ by taking $\col(uv)=\col(wy)=1$, $\col(vw)=2$, $\col(uw)=\col(vx)=3$ and $\col(wx)=\col(vy)=4$.
If $xy\in E(G)$, then denote $N(x)=\{v,w,y,x_1\}$ and $N(y)=\{v,w,x,y_1\}$. We shall also assume that $x_1\not=y_1$ because otherwise $G\simeq G[\{u,v,w,x,y,x_1\}]$ by the 2-connectivity of $G$, which admits a 4-coloring. Now we remove $u,v$ and $w$ from the diagram of $G$. Denote by $G''$ the resulting diagram. Then $G''$ is a PO-graph so that both $x$ and $y$ has degree two in $G''$. Since the diagram of $G$ minimizes the number of crossings, $xx_1$ does not cross $yy_1$ in $G$ (and thus in $G''$). Denote by $G'$ the graph obtained from $G''$ by contracting the edge $xy$. From the above arguments, one can see that $G'$ is still a PO-graph with $E(G)\setminus E(G')=\{uv,uw,vw,vx,wx,vy,wy,xy\}$. Furthermore, by the minimality of $G$, $G'$ admits a 4-coloring $\phi$ with $\phi(xx_1)\neq \phi(yy_1)$. Suppose that $\phi(xx_1)=1$ and $\phi(yy_1)=2$. Then we can construct a 4-coloring $\col$ of $G$ by taking  $\col(uw)=\col(vy)=1$, $\col(uv)=\col(wx)=2$, $\col(vw)=\col(xy)=3$ and $\col(vx)=\col(wy)=4$. Second, assume that one of $x$ and $y$, say $x$ wlog., has degree at most three. If $d(x)\leq 2$, then it is easy to see that $G\simeq G[\{u,v,w,x,y\}]$ by the 2-connectivity of $G$, which admits a 4-coloring. If $d(x)=3$, then denote $N(x)=\{v,w,x_1\}$. Consider the PO-graph $G'=G-uv$, which admits a 4-coloring $\phi$ by the minimality of $G$. If $A:=S\setminus \{\phi(uw),\phi(vw),\phi(vy),\phi(vx)\}\neq \emptyset$ (recall that $S=\{1,2,3,4\}$), then let $\col(uv)\in A$. Otherwise, assume that $\phi(uw)=1$, $\phi(vw)=2$, $\phi(vy)=3$ and $\phi(vx)=4$ wlog. It follows that $\phi(wx)=3$ and $\phi(wy)=4$. If $\phi(xx_1)=1$, then we can construct a 4-coloring of $G$ by recoloring $vx$ and $uw$ with 2, recoloring $vw$ with 1 and coloring $uv$ with 4. If $\phi(xx_1)=2$, then we can again construct a 4-coloring of $G$ by recoloring $vx$ with 1 and coloring $uv$ with 4.

If $G\supseteq G_{13}$, then we shall assume that $d(x)=\Delta(G)=4$ by Lemma \ref{ed}. Denote the fourth neighbor of $x$ by $x_1$ and meanwhile assume that $d(y)=4$ and $N(y)=\{v,w,y_1,y_2\}$ wlog. Then the PO-graph $G'=G\setminus \{u,v,w\}$ admits a 4-coloring $\phi$. Wlog. assume that $\phi(xx_1)=1$. Construct a 4-coloring $\col$ of $G$ as follows.
If $1\in C^1_{\phi}(y)$ (suppose $\phi(yy_1)=1$ and $\phi(yy_2)=2$ wlog.), then let $\col(vw)=1$, $\col(uv)=\col(wx)=2$, $\col(vx)=\col(wy)=3$ and $\col(ux)=\col(vy)=4$. If $1\not\in C^1_{\phi}(y)$ (suppose $\phi(yy_1)=2$ and $\phi(yy_2)=3$ wlog.), then let $\col(vy)=1$, $\col(ux)=\col(vw)=2$, $\col(uv)=\col(wx)=3$ and $\col(vx)=\col(wy)=4$.

If $G\supseteq G_{16}$, then we shall assume that $d(x)=d(y)=\Delta(G)=4$ by Lemma \ref{ed}. Denote the fourth neighbor of $x$ and $y$ by $x_1$ and $y_1$ respectively. Then the PO-graph $G'=G\setminus \{u,v,w,z\}$ admits a 4-coloring $\phi$. Construct a 4-coloring $\col$ of $G$ as follows.
If $\phi(xx_1)=\phi(yy_1)=1$, then let $\col(vw)=1$, $\col(ux)=\col(vz)=\col(wy)=2$, $\col(wx)=\col(vy)=3$ and $\col(uw)=\col(vx)=\col(yz)=4$. If $1=\phi(xx_1)\neq \phi(yy_1)=2$, then let $\col(vz)=\col(wy)=1$, $\col(ux)=\col(wy)=\col(vz)=2$, $\col(wx)=\col(vy)=3$ and $\col(uw)=\col(vx)=4$.

If $G\supseteq G_{17}$, then we shall assume that $d(x)=d(y)=\Delta(G)=5$ by Lemma \ref{ed}. Then the PO-graph $G'=G\setminus \{u,v,w,z,a\}$ admits a 5-coloring $\phi$. Construct a 5-coloring $\col$ of $G$ as follows.
If $C^1_{\phi}(x)=C^1_{\phi}(y)=\{1,2\}$, then let $\col(uw)=\col(av)=1$, $\col(wz)=\col(uv)=2$, $\col(xz)=\col(vw)=\col(ay)=3$, $\col(wx)=\col(vy)=4$ and $\col(vx)=\col(wy)=5$. If $|C^1_{\phi}(x)\cap C^1_{\phi}(y)|=1$ (suppose $C^1_{\phi}(x)=\{1,2\}$ and $C^1_{\phi}(y)=\{1,3\}$ wlog.), then let $\col(vw)=1$, $\col(wy)=\col(av)=2$, $\col(wz)=\col(vx)=3$, $\col(wx)=\col(uv)=\col(ay)=4$ and $\col(xz)=\col(uw)=\col(vy)=5$. If $|C^1_{\phi}(x)\cap C^1_{\phi}(y)|=0$ (suppose $C^1_{\phi}(x)=\{1,2\}$ and $C^1_{\phi}(y)=\{3,4\}$ wlog.), then let $\col(vw)=\col(ay)=1$, $\col(wz)=\col(vy)=2$, $\col(vx)=\col(uw)=3$, $\col(wx)=\col(av)=4$ and $\col(xz)=\col(uv)=\col(wy)=5$.
\end{proof}

\begin{figure}\label{sppog}
\begin{center}
  \includegraphics[width=14cm,height=6cm]{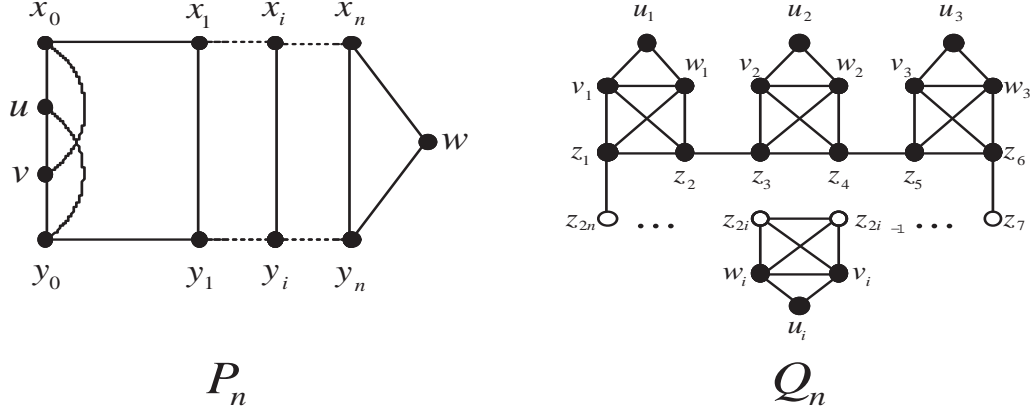}\\
  \caption{Special pseudo-outerplanar graphs}
\end{center}
\end{figure}

\begin{theorem}
For each integer $n\geq 1$, there exists a $2$-connected pseudo-outerplanar $G$ with order $2n+5$ and $\Delta(G)=3$ so that $\chi'(G)=\Delta(G)+1$.
\end{theorem}

\begin{proof}
Let $C=x_0\cdots x_n w y_n \cdots y_0 v u x_0$ be a cycle. We add edges $x_i y_i$ for all $1\leq i\leq n$ and add another two edges $x_0 v$ and $y_0 u$ to $C$. Denote the resulting graph by $P_n$ (See Figure 4). One can easily check that $P_n$ is a 2-connected pseudo-outerplanar graph with $|P_n|=2n+5$ and $\Delta(P_n)=3$. If $P_n$ has a $3$-coloring $\phi$, then we shall have $\phi(x_0 v)=\phi(y_0 u)$ and $\phi(x_0 u)=\phi(y_0 v)$ (otherwise we cannot color $uv$ properly). Thereby we would deduce that $\phi(x_ix_{i+1})=\phi(y_iy_{i+1})$ for all $0\leq i\leq n-1$ and then $\phi(x_n w)=\phi(y_n w)$. This final contradiction implies that $\chi'(P_n)=\Delta(P_n)+1=4$.
\end{proof}

\begin{theorem} \label{linear arborocity}
Let $G$ be a pseudo-outerplanar graph. If $\Delta(G)=3$ or $\Delta(G)\geq 5$, then $la(G)=\lceil\frac{\Delta(G)}{2}\rceil$.
\end{theorem}

\begin{proof}
Since conjecture \ref{lac} has already been proved for planar graphs and every PO-graph is planar (cf. Section 1), this theorem holds trivially when $\Delta(G)$ is odd. Thus in the following we assume that $\Delta(G)\geq 6$ and $\Delta(G)$ is even. For brevity we write $k=\frac{\Delta(G)}{2}$. Suppose for a contradiction that there exists a minimal (in terms of the size) pseudo-outerplanar graph $G$ that has no $k$-coloring. One can easily observe that $G$ is $2$-connected and la-critical. By Theorem \ref{2con} and Lemma \ref{la}, $G$ contains the configuration $G_3$.

If $xy\not\in E(G)$, then by (b) of Lemma \ref{2con}, $G'=G\setminus \{v\}+xy$ is still a PO-graph. Thus by the minimality of $G$, $G'$ admits a $k$-coloring $\phi$. Now we can construct a $k$-coloring $\col$ of $G$ by taking $\col(vx)=\col(vy)=\phi(xy)$ and $\col(e)=\phi(e)$ for every $e\in E(G)\cap E(G')$.

If $xy\in E(G)$, then consider the PO-graph $G'=G\setminus \{v\}$, which has a $k$-coloring $\phi$ by the minimality of $G$. It is easy to see that $|C_{\phi}^1(x)|=|C_{\phi}^1(y)|=1$, since $d(x)=d(y)=\Delta(G)=2k$ by Lemma \ref{la}. We now construct a coloring $\col$ of $G$ by taking $\col(vx)\in C_{\phi}^1(x)$, $\col(vy)\in C_{\phi}^1(y)$ and $\col(e)=\phi(e)$ for every $e\in E(G)\cap E(G')$. If $C_{\phi}^1(x)\neq C_{\phi}^1(y)$, then it is easy to see that $\col$ is a $k$-coloring. If $C_{\phi}^1(x)=C_{\phi}^1(y)$, then $\col(vx)=\col(vy)$ and $\col$ is also a $k$-coloring unless $\col(xy)=\col(vx)$ or $\col(ux)=\col(uy)=\col(vx)$. If $\col(xy)=\col(vx)$, then $\col(vx)\not\in \{\col(ux),\col(uy)\}$ and thus we can exchange the colors on $ux$ and $vx$. One can easy to check that the resulting coloring of $G$ is a $k$-coloring. If $\col(ux)=\col(uy)=\col(vx)$, then we recolor $xy$ with $\col(vx)$ and recolor both $vx$ and $uy$ with $\col(xy)$. The resulting coloring of $G$ is also a $k$-coloring.
\end{proof}

\begin{theorem}
For each integer $m\geq 1$, there exists a $2$-connected pseudo-outerplanar $G$ with order $10m+5$ and $\Delta(G)=4$ so that $la(G)=\lceil\frac{\Delta(G)}{2}\rceil+1$.
\end{theorem}

\begin{proof}
Let $C=z_1\cdots z_{2n} z_1$ be a cycle and $T_i=u_i v_i w_i u_i$$(1\leq i\leq n)$ be triangles. Suppose that they are pairwise disjoint. Now for each $1\leq i\leq n$, add fours edges $v_i z_{2i-1}$, $v_i z_{2i}$, $w_i z_{2i-1}$ and $w_i z_{2i}$. Denote the resulting graphs by $Q_n$ (See Figure 4). One can easily check that $Q_n$ is a 2-connected pseudo-outerplanar graph with $\Delta(Q_n)=4$. Consider the graph $Q_{2m+1}$($m\geq 1$). It is trivial that $|Q_{2m+1}|=10m+5$ and $la(Q_{2m+1})\leq 3$ by Lemma \ref{la}. If $Q_{2m+1}$ has a 2-coloring $\phi$, then we shall have $\phi(z_{2i-2}z_{2i-1})\neq \phi(z_{2i}z_{2i+1})$ for all $1\leq i\leq 2m+1$, where $z_0=z_{4m+2}$ and $z_{4m+3}=z_{1}$ (otherwise we cannot properly color the set of edges $\{u_iv_i,v_iw_i,w_iu_i, v_i z_{2i-1}, v_i z_{2i}, w_i z_{2i-1}, w_i z_{2i}\}$ for some $i$). However, the size of the set $\{z_2z_3,z_4z_5,\cdots,z_{4m+2}z_1\}$ is $2m+1$, which is odd, but there are only two colors that can be used in $\phi$. This final contradiction implies that $la(Q_{2m+1})=\lceil\frac{\Delta(Q_{2m+1})}{2}\rceil+1=3$.
\end{proof}

\section*{Acknowledgement}

The authors thank the referees for many helpful comments and suggestions, which have greatly improved the presentation of the results in this paper, and would also like to acknowledge the editors for pointers to relevant literature and phraseological comments.

\end{document}